\documentclass[reqno,10pt]{amsart}
\usepackage{amssymb}
\usepackage{amsmath}
\usepackage{amscd}
\usepackage{amsthm, amsfonts,  amsmath, amssymb}
\usepackage{url}
\usepackage{amsmath, amssymb, amsthm, latexsym, enumerate}
\usepackage[top=2.9cm, bottom=2.5cm, left=3cm, right=3cm]{geometry}
\usepackage{geometry}
\usepackage{etoolbox}
\usepackage[all]{xy}
\usepackage{romannum}

\DeclareFontEncoding{OT2}{}{} 

\newtheorem{prop}{Proposition}[section]
\newtheorem{definition}[prop]{Definition}

\newtheorem{conj}[prop]{Conjecture}

\newtheorem{lem}[prop]{ Lemma}
\newtheorem{thm}[prop]{Theorem}
\newtheorem{cor}[prop]{Corollary}

\newtheorem{remar}[prop]{Remark}

\newcommand{\ZZ}{{\mathbb Z}}

\newcommand{\one}{ \Romannum {1}}
\newcommand{\two}{\Romannum {2}}
\newcommand{\three}{ \Romannum {3} }
\begin{document}
\title[On zero-sum subsequences of length exp(G)]{\small {On zero-sum subsequences of length exp(G)}
}

\title{On zero-sum subsequences of length exp(G)}
\author{Karthikesh}
\address{KARTHIKESH
\newline
\ \ INDIAN INSTITUTE OF SCIENCE EDUCATION AND RESEARCH, TRIVENDRUM, INDIA.}
\email{karthikeshkh214@iisertvm.ac.in}

\author{SriLakshmi Krishnamoorty}
\address{SRILAKSHMI KRISHNAMOORTHY
 \newline
INDIAN INSTITUTE OF SCIENCE EDUCATION AND RESEARCH, TRIVENDRUM, INDIA.}
\email{srilakshmi@iisertvm.ac.in}

\author{Umesh Shankar}
\address{UMESH SHANKAR
\newline
\ \ INDIAN INSTITUTE OF SCIENCE EDUCATION AND RESEARCH, TRIVENDRUM, INDIA.} 
\email{umeshshankar14@iisertvm.ac.in}

\begin{abstract}
Let $G$ be a finite abelian group.  Let $g(G)$ be the smallest positive integer $t$ such that every subset of
cardinality $t$ of the group $G$  contains a subset of cardinality $\mathrm{exp}(G)$ whose sum
is zero. In this paper, we show that if X is a subset of $\mathbb{Z}^2_{2n}$ with cardinality $4n+1$ and $2n$ or $2n-1$ elements of $X$ have the same first coordintes, then $X$ contains a zero sum subset. As an application of our results we prove that $g(\mathbb{Z}^2_6) = 13.$ This settles
Gao-Thangaduri's conjecture for the case $n=6.$ We also prove some results towards the general even $n$ cases of the conjecture. 
\end{abstract}
\maketitle
\section{Introduction}\label{Introduction}
Let $G$ be a finite abelian group. We know that $G \simeq \mathbb{Z}_{n_1} \oplus \mathbb{Z}_{n_2}
\oplus \cdots \mathbb{Z}_{n_d}$ with $1 <  n_1 < n_2  < \cdots < n_d,$ where $n_d = \mathrm{exp}(G).$
\begin{definition}
By $g(G)$ we denote the smallest positive integer $t$ such that every subset $X$ of $G$ with
$|X| \geq t$ contains a subset $S$ of $X$ with $|S| = \mathrm{exp}(G)$ which satisfies
$$\sum_{s \in S} s = 0_G.$$
\end{definition}
Kemnitz 
 \cite{K82} proved that $g(\mathbb{Z}^2_p) = 2p - 1$ for $p = 3, 5, 7.$
 It is known from the work of Gao-Thangadurai that  $g(\mathbb{Z}^2_p) = 2p - 1$ for primes
 $ p \geq 47$ and $g(\mathbb{Z}^2_4)= 9.$ 
For odd $n$, consider the subset $A = \{ (0,0), (0,1),\cdots,(0,n-2), (1,1), (1,2),\cdots,(1,n-1) \}$
of $\mathbb{Z}^2_n.$ Then $|A| =  2n - 2$ and $A$ does not contain any zero-sum subset of cardinality
$n.$  Hence $g(\mathbb{Z}^2_n) \geq 2n - 1.$ For even $n$, consider the subset $B = \{ (0,0), (0,1),\cdots, (0, n-1), (1,0), (1,1), (1,2),\cdots,(1,n-1) \}$
of $\mathbb{Z}^2_n.$ Then $|B| =  2n$ and $B$ does not contain any zero-sum subset of cardinality
$n.$  Hence $g(\mathbb{Z}^2_n) \geq 2n +1.$ 
\\
Gao and Thangadurai conjectured the following \cite{GT04}
\begin{conj}\label{main-conjecture}
\begin{equation*}
g(\mathbb{Z}^2_n)= 
\begin{cases}
2n-1 & \text{if n is odd,} \\
2n+1 &  \text{if $n$ is even.} 
\end{cases}
\end{equation*}
\end{conj}
 Keminitz studied the constant for the general group $G=\mathbb{Z}^d_n$ \cite{K82}, \cite{K83}.
 Several authors followed his work and 
studied the constant for the group $G=\mathbb{Z}^d_n,$ for all $n, d.$ ( \cite{B82}, \cite{BB82}, \cite{G97},
\cite{BE02}, \cite{EFLS02}, \cite{M95}, \cite{R00}, \cite{ST02} ).
In this paper, we prove that $g(\mathbb{Z}^2_6) = 13.$ In Section 3 of this paper, we show that ( Theorems \ref{Main Theorem 1}, \ref{Main Theorem 2}) if $X$ is a subset of $\mathbb{Z}^2_{2n}$ with cardinality $4n+1$ and $2n$ or $2n-1$ elements of $X$ have the same first coordinats, then $X$ contains a zero sum subset. In Section 4, as an application of our results, we prove Gao-Thangadurai's conjecture for the case $n=6.$ We also prove some results towards the general even $n$ cases.

\section{Preliminaries}
\begin{definition} Let $ A \subseteq \mathbb{Z}_{n},$  
For any natural number $m \leq n,$ the (unrestricted) $m$-fold sumset of $A,$
denoted by $mA,$ is the collection of all elements of $ \mathbb{Z}_{n}$ that can be written as the sum of $m$(not necessarily distinct) elements of $A,$ and the $m$-fold restricted sumset of A, denoted by $A{\psi}^{m-1}A,$ consists of the elements of  $\mathbb{Z}_{n}$ that can be written as the sum of $m$ distinct elements of $A.$
%
\end{definition}
\begin{definition} We say that a set $X= \left\lbrace a,b \right\rbrace $ $\subseteq$ $\mathbb{Z}_{2n}$  is of form $\Romannum{1},$ if $b=a+3.$ We say that a set $Y$ $\subseteq$  $\mathbb{Z}_{2n}$ is of form  $\Romannum{2},$ if all the elements of $Y$ are of same parity.  If a set is neither form $\Romannum{1}$ nor form $\Romannum{2}$, we say that it is of form $\Romannum{3}.$
\end{definition}
\begin{remar}
For $n=3,$ a set which is of form $\one$ can never be of form $\two$ and vice versa.
\end{remar}
We state the following theorem by Erdos-Ginzburg-Ziv Theorem  \cite{EGZ61}.
\begin{thm}\label{EGZ}
Each set of $2n-1$ integers contains some subset of $n$ elements the sum of which is a multiple of $n.$ 
\end{thm}
\begin{lem}
\label{gp}
Let $ A = \left\lbrace a_1, a_2 \right\rbrace $ and  $B = \left\lbrace b_1, b_2 \right\rbrace $  are subsets of $ \mathbb{Z}_{2n}.$ Then
$|A + B| \geq 2$ with equality iff $a_1 \equiv a_2 \pmod{n} $ and $b_1 \equiv b_2 \pmod{n}.$ 
\end{lem}
\begin{proof}
We know that $a_1+b_1,a_2+b_1$ are distinct as $a_1 \neq a_2,$ hence $|A + B| \geq 2.$ We also know that $a_1+b_1$ and $a_1+b_2$ are distinct as  $b_1 \neq b_2.$    We can see that $ |A+B|=2$ iff $ a_1+b_1 \equiv a_2+b_2 \pmod{2n}$ and $ a_2+b_1 \equiv a_1+b_2 \pmod {2n}$  iff  $a_1 \equiv a_2 \pmod{n}$ and $b_1 \equiv b_2 \pmod{n}.$
\end{proof}

\begin{lem}Let $S \subseteq \mathbb{Z}_{2n}^2$ with $|S| = 4n+1.$ Let the elements of $S$ are in row $a_1,$ $a_2, \cdots, a_{t}$ be the elements of $S$ for some $t \leq 2n.$  Let $A_i = \{ y / (a_i,y) \in S \}$ for all $1\leq i \leq t.$
Suppose $k_1+ k_2 +\cdots+k_{t}=2n,$ $k_1 a_1 + k_2 a_2+ \cdots + k_t a_t = 0$ for some $1\leq k_i \leq 2n$ and $0 \in A_1 {\psi}^{k_1-1}A_1+
A_2 {\psi}^{k_2-1} A_2+ \cdots + A_{t} {\psi}^{k_{t}-1 } A_{t},$  then
there exists $2n$ elements in the set $S$ whose sum is zero.
\end{lem}
\begin{proof}
We use the convention that if $k_i =0,$ then $A_i {\psi}^{{k_i}-1} A_i$ is an empty set. The equation $k_1 a_1 + k_2 a_2+ \cdots + k_{t}a_t = 0.$
Also, $0 \in A_1 {\psi}^{k_1-1}A_1+A_2 {\psi}^{k_2-1} A_2+ \cdots + A_{t} {\psi}^{k_{t} -1 } A_{t},$ which implies that there exists a choice of $k_i$ elements in $A_i,$ say $b_1^i, b_2^i \cdots, b_{k_i}^i,$ such that, $$ \sum_{1 \leq i \leq {t}}  \sum_{1 \leq j \leq {k_i}} b_{j}^i = 0.$$ 
Hence the $2n$ elements $(a_i, b_j^i)$ where $1\leq i \leq t, 1 \leq j \leq {k_i} $ adds upto $(0,0).$
\end{proof}
\begin{lem}\label{main}
 Let $A,B \subseteq \mathbb{Z}_n$. If $|A|+|B| > n,$ then $0 \in A+B.$ 
\end{lem}
\begin{proof}
 Let $|A| = i$ and $|B|=j$ and $i+j >n .$ Let $A = \left\lbrace a_1,a_2,...,a_i \right\rbrace $ and $B = \left\lbrace b_1,b_2,...,b_j \right\rbrace $. Define $ -A = \left\lbrace -a_1,-a_2,...,-a_i \right\rbrace $. It is clear that $|-A|=i$. We have,
 $$ |-A|+|B|-|(-A) \cap B| = |(-A) \cup B| \leq n $$
 Since $|A|+|B| > n $, $(-A) \cap B$ is non empty.  Hence there exists $a_l$ in $A$ such that $-a_l$ is in $B$. We have $ a_l+(-a_l) = 0 \in A+B .$
 \end{proof}
\section{MAIN THEOREMS}
Let the elements of $\mathbb{Z}^2_{N}$ are in $N$ rows and $N$ columns as 
$\{(0,0), (0,1),....(0,N-1), (1,0), (1,1),....\\(1,N-1),...(N-1,0), (N-1,1),....,(N-1,N-1) \}.$ We denote
$p_k ( \leq m, l ) $ the number of partitions of $l$ into at most $k$ parts, each less than or equal to $m.$
There are ${  N ^2  \choose {2N+1}  }$ subsets of $\mathbb{Z}_N^2$ with cardinality $2N+1.$
Given $X \subset \mathbb{Z}_N^2,$ with $|X| = 2N +1,$ elements of $X$ can be distributed in $N$ rows
and $N$ columns in $p_N ( \leq N, 2N+1 )$ ways.
\begin{lem}
\label{gen}
 Let $a_1,a_2,...,a_{2N+1}$ be a sequence in $\mathbb{Z}_N$ such that (i) $a_1=a_2= \cdots =a_N,$\\(ii) $ a_i \neq a_1$ $\forall$ $i \geq N+1$ and (iii) $a_{N+1},\cdots,a_{2N+1}$ have atleast two distinct elements. Then we have a subsequence of length $N,$ say $b_1,b_2,\cdots,b_N,$ such that $\sum_{i=1}^n b_i = 0,$  with at least one of the  $b_i $ is  $a_1$ and not all $b_i'$s are $a_1.$ 
 \end{lem}
 \begin{proof}
WLOG Let $a_{2N} \neq a_{2N+1}.$ Consider the sequence $a_2,\cdots , a_{2N}.$ By Theorem \ref{EGZ} we know there  exists a $N$-subsequence whose sum is zero. Note that not all elements in such $N$-subsequence are $a_1.$ If this subsequence contains $a_1,$ then we are done. Otherwise it must be $a_{N+1},\cdots , a_{2N}.$ This gives that $a_{N+1}+a_{N+2}+\cdots+a_{2N} \equiv 0 \pmod {N}.$ By  Theorem \ref{EGZ}, the sequence $a_2, a_3, \cdots , a_{2N-1}, a_{2N+1}$ contains a $N$-subsequence whose sum is zero. Note that not all elements in such $N$-subsequence are $a_1.$ If this contains $a_1,$ then we are done. Otherwise, it must be $a_{N+1}, \cdots ,a_{2N-1}, a_{2N+1}.$ Hence $a_{N+1}+a_{N+2}+ \cdots + a_{2N-1}+a_{2N+1} \equiv 0 \pmod{N}.$ Equating the two zero sums implies that $ a_{2N} = a_{2N+1},$ which is a contradiction. Hence there is an N-subsequence containing $a_1.$   
\end{proof}
\begin{thm} \label{Main Theorem 1}Suppose $(a_1,b_1),(a_2,b_2),...(a_{2N+1},b_{2N+1})$ is a subset in ${\mathbb{Z}_N^2}$ such that it contains all elements from a particular row that is, $a_1=a_2= \cdots = a_N.$ Then there exists an $N$ sum which adds upto $(0,0).$
\end{thm}
\begin{proof}
Since $a_1 = a_2 = ..... = a_N,$ the elements $a_{N+1}, \cdots, a_{2N+1}$ are not all equal. Consider the sequence of first corodinates. Then by Lemma $\ref{gen},$ there exists an $N$-subsequence whose sum is zero and it contains $a_1.$   Now consider the second coordinates. Suppose $B$ denote the sum of second coordinates of the remaining $N-1$ sums (other than $a_1$). Let $A = \{ b_1, b_2,\cdots, b_N\}.$  Since $A = \mathbb{Z}_N,$  we have $ |A \psi B | \geq |A| = N$. Hence there exists a zero in $A \psi B$ and therefore we have a $N$ sum which is zero. 	 
\end{proof}
\begin{prop}\label{prop1}
Let $n \in \mathbb{N}$ and let $A \subset \mathbb{Z}_{2n}$ such that $|A| \geq n+2$ then $A \psi A = Z_{2n}.$
\end{prop}
\begin{proof} Consider an element $x$ in $Z_{2n}$.  
If $x=2k+1$ for some $k \in Z_{2n}$, define the sets $T_{x,1}' = \left\lbrace k+a : 1 \leq a \leq n \right\rbrace$ and $T_{x,2}' = \left\lbrace (k-a)+1 : 1 \leq a \leq n \right\rbrace $. Note that $ 2k+1 \in T_{x,1}' + T_{x,2}' ,$  $T_{x,1}' \cup T_{x,2}' = Z_{2n} $ and $ |T_{x,1}'| = |T_{x,2}'| = n $. There are $n$ pairs $\{ ( k+a, (k-a)+1 ) : 1 \leq a \leq n \}$ in $T_{x,1}' \cup T_{x,2}' = Z_{2n} .$  Then choosing $n+1$ elements from $A \subseteq  T_{x,1}' \cup T_{x,2}' = Z_{2n} $  guarantees the existence of $x$ in $A \psi A,$ since by pigeonhole principle, there exists atleast one pair  $( k+a, (k-a)+1 )$ such that both $k-a$ and $(k-a)+1$ belongs to $A$ for some $ 1 \leq a \leq n.$ The sum
is  $k+a + (k-a)+1 = 2k+1=x.$\\
If $x=2k$ for some $k \in Z_{2n}$, define the sets $T_{x,1} = \left\lbrace k+a : 0 \leq a \leq n \right\rbrace  $ and $T_{x,2} = \left\lbrace k-a : 0 \leq a \leq n \right\rbrace  .$ Note that $ 2k \in T_{x,1} + T_{x,2}, $  $T_{x,1} \cup T_{x,2} = \mathbb{Z}_{2n} $ and $ |T_{x,1}| = |T_{x,2}| = n+1$. Then choosing $n+2$ elements from $A$ guarantees the existence of $x$ in  $A \psi A,$ since by pigeonhole principle, there exists atleast one pair which sums up to $2k.$
\end{proof} 
\textbf{ Claim:} $A \psi^{k-1} A=Z_{2n}$ whenever $|A| \geq n+k$.
We proved it for $k=2.$ We provide a simple argument to obtain a weak lower bound for $|A|$ so that $A \psi^{k-1} A=Z_{2n}$ by induction. Assume that if $|A|\geq l,$ then $A \psi^{n-2} A = Z_{2n}.$ Let $|A|\geq l+1.$
Write $A={a}\cup (A-{a})$. Hence, $|A-{a}|\geq l$.
We know that $A \psi^{n-2} A+{a} \subset A \psi^{n-1} A $, but ${n-1}^{A-{a}}=Z_{2n}$, by induction. And by monotonicity of residue class addition, we get $A \psi^{n-1} A$ is $Z_{2n}$. This proves our claim.
\begin{remar}
We have shown that $A \psi^{k-1} A=Z_{2n}$ whenever $|A| \geq n+k$. Also, we can set up a bijection between $A \psi^{k-1} A$ and $A \psi^{|A|-k-1} A$ with the map $a \mapsto \left(\sum _{x\in A}  x\right)-a $. 
By these two facts, we can provide slightly stronger bounds. For example, 
Let $|A|=2n-1$. We get $A \psi^{k-1} A=Z_{2n}$ for $2\leq k \leq n-1$ and  $A \psi^{k-1} A=Z_{2n}$ for $2\leq k \leq 2n-3$. 
\end{remar}
\begin{definition}
 Let a subsequence of sequence of length $2n$ be of acceptable length if the length of the subsequence is not $2$ or $2n$.
\end{definition} 
\begin{lem}\label{non zero lemma}Among $2n+2$ non zero elements in $Z_{2n}$ not more than $2n$ elements are equal, there exists a zero sum subsequence of acceptable length.
\end{lem}
\begin{proof}
\textbf{Case 1: Suppose A doesn't contain a pair of inverses}
Let A be $a_1, \cdots ,a_{2n+2}$ such that $a_{2n} \neq a_{2n+2}$. Choose the subsequence $a_1,\cdots, a_{2n}$.
If $a_1+\cdots+ a_{2n}=0 $, then replace $a_{2n}$ by $a_{2n+1}$.
Now, $a_1+\cdots+ \hat{a_{2n-1}}+a_{2n+1}\neq 0 $.
But since $D(Z_{2n})=2n$, we must have a zero sum subsequence of acceptable length.
\textbf{Case 2:
Suppose $A$ contains atleast a pair of inverses.} $A$ must contain atleast three different elements.
If not, then if $a, b$ are the inverses, we have the acceptable length subsequence $a,a,b,b$ which adds upto zero.
WLOG, let the inverses be in $a_1$ to $a_{2n+1}$ and the third element be $a_{2n+2}$. We can assume that $A$ doesn't have two copies of $a$ and $b$ simultaneously. Otherwise, $a,a,b,b$ exists.
WLOG, we say that $b$ appears only once in $A.$
\textbf{Case 2.a: $A$ contains only 3 different elements}
Now $A$ looks like $b,a,\cdots,a,c,\cdots,c$.
Pick the $2n$ element sequence to be $A$ without one copy of $b$ and $a.$ If this sequence sums upto zero, then pick the $2n$ sequence obtained by removing $b$ and $c$ from $A.$ This can't add upto zero and therefore, has a subsequence of acceptable length that sums to zero.
Note: The argument might not work when $A$ is of the $b,c,a,\cdots,a$ where $a$ appears $2n$ times, but clearly the $2n$ sequence $c,a,\cdots,a$ doesn't add upto zero.

\textbf{Case 2.b: $A$ has more than 4 elements}
Say, $a,b,c,d$ and others.
$A$ looks like \\
$b,a,\cdots,a,c,\cdots ,c, d,\cdots,d,others$.
Pick the $2n$ element sequence to be $A$ without one copy of $b$ and $a.$ If this sequence sums upto zero, then pick the $2n$ sequence obtained by removing $b$ and $c$ from $A.$ This sequence can't sum upto zero. If it has a two element subsequence that sums upto zero, say $e,f,$ then clearly $a+b+e+f=0.$ This completes our proof.
\end{proof}
\begin{thm}\label{Main Theorem 2}
 Let A be a $4n+1$ length sequence in $Z_{2n}\times Z_{2n}$ such that $2n-1$ elements are from a particular row, then there exists a $2n$ subsequence of A that sums up to zero.
 \end{thm}
\begin{proof}
 Let $A$ be the sequence $x_1, \cdots, x_{4n+1}$ where $x_i=(a_i,b_i)$. Let $A'$ be the sequence of first coordinates, that is $a_i$. Let us assume that the first $2n-1$ elements are from the particular row $a,$ $a_i=a$ when $1 \leq i \leq 2n-1$. We will now construct a $2n$ element subsequence of $A'$ that adds up to zero. Construct $A''$ to be the subsequence $a_i-a$'s where $2n \leq i \leq 4n+1$. This is $2n+2$ non zero element subsequence which satisfies the criterion in the lemma \ref{non zero lemma}. Therefore, this contains a zero sum subsequence of acceptable length. Let this zero sum subsequence be ${{{a^{'} }_i}_j}$ where $1\leq j \leq k$. But  ${{{a^{'} }_i}_j}={a_i}_j-a$. Since this sequence is a zero sum sequence $\sum_{j=1}^k {a_i}_j=ka$.
Therefore, a new $2n$-length sequence $y_i=a$ for $1\leq i \leq 2n-k$, $y_{2n-k+j}={a_i}_j$ where $1\leq j \leq k$. By way of construction, $\sum_i y_i=0 $.\\
Look at set of all possible second coordinates sums of sub-sequences of $A$ whose first coordinates are $y_i$ be $B$. 
This is of the form $ B \psi^{2n-k-1} B + D$, for some set $D.$ But from our previous lemma, we know that 
$B \psi^{2n-k-1} B = Z_{2n}$ and therefore, it contains a zero. This completes our proof.
\end{proof}

\section{Application}
We apply our main theorems to prove the following theorem, which settles Gao-Thangadurai's conjecture for
$n=6.$ There are 23,107,896,00 subsets of $\mathbb{Z}_{6}^2$ with cardinality 13. Given 
$X \subseteq \mathbb{Z}_{6}^2,$ with $|X|=13,$ elements of $X$ can be distributed in six rows
in $42 = p_6(\leq 6, 13)$ ways, given by

{\bf 1.} 6+6+1 (six elements in one row, six elements in another row and one element in another row),\\
{\bf 2.} 6+5+2, {\bf 3.} 6+5+1+1, {\bf 4.} 6+4+3, {\bf 5.} 6+4+2+1, {\bf 6.} 6+4+1+1+1, {\bf 7.} 6+3+3+1, {\bf 8.}  6+3+2+2, {\bf 9.} 6+3+2+1+1, {\bf 10.} 6+3+1+1+1+1, {\bf 11.} 6+2+2+2+1, {\bf 12.}  6+2+2+1+1+1, {\bf 13.}  5+5+3, {\bf 14.} 5+5+2+1, {\bf 15.} 5+5+1+1+1, {\bf 16.} 5+4+4, {\bf 17.}  5+4+3+1, {\bf 18.} 5+4+2+2, {\bf 19.} 5+4+2+1+1, {\bf 20.} 5+4+1+1+1+1,
{\bf 21.} 5+3+3+2, {\bf 22.} 5+3+3+1+1, {\bf 23.} 5+3+2+2+1, {\bf 24.} 5+3+2+1+1+1, {\bf 25.} 5+2+2+2+2, {\bf 26.} 5+2+2+2+2+1+1, {\bf 27.} 4+4+4+1, {\bf 28.} 4+4+3+2, {\bf 29.} 4+4+3+1+1, {\bf 30.} 4+4+2+2+1, {\bf 31.} 4+4+2+1+1+1, {\bf 32.} 4+3+3+3, {\bf 33.} 4+3+3+2+1, {\bf 34.} 4+3+3+1+1+1, {\bf 35.} 4+3+2+2+2, {\bf 36.} 4+3+2+2+1+1, {\bf 37.} 4+2+2+2+2+1, {\bf 38.} 3+3+3+3+1, {\bf 39.} 3+3+3+2+2, {\bf 40.} 3+3+3+2+1+1, {\bf 41.} 3+3+2+2+2+1, {\bf 42.} 3+2+2+2+2+2.
\begin{thm}\label{Main Theorem}
Given any $X \subseteq \mathbb{Z}_{6}^2,$ with $|X| = 13,$ there exists a subset of $S$ of $X$ with
$|S| = 6,$ and $\sum_{s \in S} s = (0,0).$
\end{thm}
\textbf{  \ \ \ \ \  \  \ \ \ \ \ \ \ \ \  3+3+4 \ (the partitions of $10$ into three parts)  }
\begin{prop}\label{Main Proposition 1}
Let $a,b,c$ denote the corresponding rows in which 3,3,4 elements are distributed respectively.
There is a zero in the sequence $2(a+b+c), 2a+b+3c, 2b+c+3a, a+2b+3c, 2a+c+3b,a+2c+3b,b+2c+3a.$ 
\end{prop}
\begin{proof}
Suppose none of the elements $2(a+b+c), 2a+b+3c, 2b+c+3a, a+2b+3c, 2a+c+3b$ are divisible by 6.
Then we show that $a+2c+3b$ or $b+2c+3a$ is divisible by 6. Consider the sequence $aabbc$.  
By Theorem \ref{EGZ}, we have atleast one element in the sequence $T_1 : a+b+c,2a+b,2b+c,a+2b,2a+c$ is divisible by $3.$ Consider the sequences $abbcc$. By Theorem \ref{EGZ}, atleast one element in $T_2 : a+b+c, a+2b,a+2c,b+2c,2b+c$ is divisible by $3.$   
\\
Let us denote the sequence $a+3c,c+3a,c+3b,b+3c$  by $S$. 
Suppose there is atleast one even in $S,$ then we have an element in $T_1,$ which is divisible by $6.$ That is a contradiction to our assumption. Hence all elements in $S$ are odd. Therefore $a+3c + c+3b$ and $b+3c+c+3a$ are even. Hence either $a+3b+2c$ or $b+3a+2c$ in $T_2$ is divisible by $6$.\\ 
We consider four elements in the sequence $2(a+b+c), 2a+b+3c, 2b+c+3a, a+2b+3c, 2a+c+3b,a+2c+3b,b+2c+3a$ and prove the proposition. Remaining three cases can be proved similarly by interchanging $a$ and $b.$
\\
Suppose $3a+2b+c = 0 ,$  this implies that if we choose 3 elements from row $a$, 2 from row $b$ and one from row $c$, the first coordinate adds upto $ 0$ modulo $6.$\\
We have $|A\psi A\psi A|=1,|B\psi B| =3,$  and $|C|=4.$  Since $|A\psi A\psi A+B\psi B|+|C|=7>6,$ \\ $ 0 \in A\psi A\psi A+B\psi B+C,$ by Lemma $\ref{main}.$
\\
Suppose $a+3b+2c = 0,$ consider the set $A+B\psi B\psi B+C\psi C.$ We have $ | B \psi B \psi B| =1  $ and therefore,
$ |A+B \psi B \psi B| = 3. $ Since 
$ |C \psi C|\geq 4$ (by Lemma $\ref{lp}$ i). We have $|A + B \psi B \psi B|+|C\psi C| > 6.$ Hence $0 \in A + B \psi B \psi B+C\psi C,$ by Lemma $\ref{main}.$
\\
 Suppose $a+2b+3c = 0. $ 
 Then  
$ |B \psi B| = 3  $ and therefore, $ |A  + B \psi B| \geq 3 $. Since $ | C \psi C \psi C| = 4  ,$  $ 0 \in A  + B \psi B+ C \psi C \psi C,$ by Lemma $\ref{main}.$
\\
Suppose $2(a+b+c)$ is zero. Then $ |B \psi B| \geq 3  $ and $ |A  \psi A| \geq 3. $ This implies $ |A \psi A  + B \psi B| \geq 3. $ Since $ |C \psi C| \geq 4 ,$  $ 0 \in A \psi A + B \psi B+C \psi C,$ by Lemma $\ref{main}.$
\end{proof}
 Out of these $42 = p_6(\leq 6, 13)$ cases, Theorem \ref{Main Theorem 1} and \ref{Main Theorem 2}, Theorem \ref{Main Theorem} is solved for the cases where 6 and 5 appears as an entry in the partition of $13.$ Cases in which 3+3+4 appears are also solved.  We show that the remaining $10$ cases of partitions of $13$ follows from Lemma 2.7 by computing the cardinalities of certain subsets. Next we provide the proofs of these $10$ cases, case by case. We need following three lemmas.
\begin{lem} 
\label{lp-sp}Let $A\subseteq  \mathbb{Z}_{6}.$ For any natural number $y < |A|,$
$$ | A \psi^y A | \geq |A|  \  \text{ if $1 \leq y < |A|-1$,} \
 | A \psi^y A | = 1 \  \text{if $y = |A|-1.$ }$$ 
\end{lem}
\begin{proof}
Suppose $A = \{a_1,a_2,...,a_l\}$. If $y=l-1,$ then  $ |A \psi A |= | \{a_1+a_2+...+a_l\}| = 1.$\\
We shall prove the inequality for the cases $ |A|=3,4,5,$ as it follows trivially for the rest of the cases in $ \mathbb{Z}_6.$\\
\textbf{Case I :} $|A|=3.$ Let $A = \{a_1,a_2,a_3\}.$ \\
We have $A \psi A = \{a_1+a_2,a_2+a_3,a_1+a_3\}$ and all of these are distinct in congruence modulo $6.$ Hence $|A \psi A|=|A|=3.$\\
\textbf{Case II : } $|A|=4.$ Let $A = \{a_1,a_2,a_3,a_4\}. $ \\
\textbf{Case II.1 : } $y=1$\\
We then have $ A \psi A = \{ a_1+a_2,a_1+a_3,a_1+a_4,a_2+a_3,a_2+a_4,a_3+a_4 \}.$ It is clear that $a_1+a_2,a_1+a_3,a_1+a_4$ are distinct in modulo $6.$ Now suppose if we assume, $| A \psi A| < |A|$. We then have \\
$a_1+a_2 = a_3 + a_4,$ $a_1 + a_3 = a_2 + a_4$ and $a_1 + a_4 = a_2 + a_3.$ Hence
we get $a_2 = a_3 + 3 $ and $a_2 = a_4 + 3 .$\\
Therefore we have a contradiction and hence $ |A \psi A| \geq |A|.$\\
\textbf{Case II.2 :} $y=2$\\
We have $ A \psi A \psi A = \{ a_1+a_2+a_3,a_1+a_2+a_4, a_1+a_3+a_4, a_2+a_3+a_4 \}.$ It is clear that $ a_1+a_2+a_3, a_1+a_2+a_4, a_1+a_3+a_4$ are distinct. WLOG, if $ a_2+a_3+a_4 \equiv a_1+a_2+a_3. $ We then have $ a_4 = a_1,$ which is a contradiction. Hence, $|A \psi A \psi A| = |A|=4.$\\
\textbf{Case III : } $|A|=5.$ Let $ A = \{a_1,a_2,a_3,a_4,a_5\} .$\\ 
\textbf{Case III.1 :} $y=1$\\
We have $A \psi A = \{a_1+a_2,a_1+a_3,a_1+a_4,a_1+a_5,a_2+a_3,a_2+a_4,a_2+a_5,a_3+a_4,a_3+a_5,a_4+a_5 \}.$ It is clear that $a_1+a_2,a_1+a_3,a_1+a_4,a_1+a_5 $ are disitint in modulo $6.$ Therefore, $|A \psi A| \geq 4.$ Now suppose  $|A \psi A| < |A|$ that is $|A \psi A|=4.$ For the sake of notations, we shall denote the tuple $(a_i,a_j)$ by $(i,j).$ Consider $(1,2).$ Then it is clear that the congruences $(2,3),(2,4),(2,5)$ can not be equal to $(1,2).$ WLOG, let $(1,2)=(3,4).$ We then have two subcases \\
\textbf{Subcase III.1.1 : } $ (1,3)=(2,4), (1,4)=(2,5), (1,5)=(2,3).$ \\
The left out congruences are $(3,5),(4,5).$ We then have $(4,5)$ to be equal to $(1,2)$ or $(1,3).$ But then we have $(4,5)=(1,2)=(3,4)$ which implies $a_5 = a_3,$ or $(4,5)=(1,3)=(2,4)$ which implies $a_5=a_2$  a contradiction.\\
\textbf{Subcase III.1.2 :} $ (1,3)=(2,5), (1,4)=(2,3), (1,5)=(2,4).$\\ The left out congruences are $(3,5),(4,5).$ We then have $(4,5)$ to be equivalent to $(1,3).$ But then we have $(4,5)=(1,3)=(2,5)$ which implies $a_4=a_2,$ again a contradiction.\\
Therefore, $|A \psi A| \geq |A|. $\\
\textbf{Case III.2 : } $y=2$  \\
$ A \psi A \psi A = \{ (1,2,3),(1,2,4),(1,2,5),(1,3,4),(1,3,5),(1,4,5),(2,3,4),(2,3,5),(2,4,5),(3,4,5) \} .$ Consider the set $ A^{'} = \{a_2,a_3,a_4,a_5 \}.$ From the previous case for cardinality equal to four, we get $|A^{'} \psi A^{'} \psi A^{'}|=4.$ Hence $|A \psi A \psi A| \geq 4$. \\
Now suppose $| A \psi^2 A| < |A|,$ that is $|A \psi^2 A| = 4.$
The elements $(2,3,4),(2,3,5),(2,4,5),(3,4,5)$ are distinct in modulo $6.$ Now we start our cases\\
\textbf{Case III.2.1 :} $(3,4,5) = (1,2,3)$ \\ 
Since $(1,2,4)$ and $(1,2,5)$ can not be congruent to $(2,4,5).$ We have $(1,2,4)=(2,3,5)$ and $(1,2,5) = (2,3,4).$ But then $(1,4,5)$ being congruent to any of the four elements in $A^{'} \psi^2 A^{'}$ leads to a contradiction.\\
\textbf{Case III.2.2 :} $(1,2,3) = (2,4,5)$\\
\textbf{Case III.2.2a :} $(1,2,4) = (3,4,5)$\\
 The left over congruences are $(2,3,4)$ and $(2,3,5).$ It is clear that $(1,2,5) = (2,3,4).$ Consider the element $(1,3,5).$ This could be congruent to $(2,3,4)$ or $(2,4,5).$ But then it implies $(1,3,5)=(1,2,5)$ or $(1,3,5)  = (1,2,3)$ which is a contradiction.\\
 \textbf{Case III.2.2b :} $(1,2,4)=(2,3,5)$\\
 We then have $(1,2,5) = (3,4,5)$ or $(1,2,5) = (2,3,4)$ \\ 
Consider the element $(1,3,4)$. In either of the above cases, $(1,3,4)$ can not be congruent $(2,3,4)$ or $(3,4,5).$ Therefore it could be congruent to $(2,3,5)$ or $(2,4,5).$ But then this implies $(1,3,4)$ is congruent to $(1,2,4)$ or $(1,2,3),$ which is a contradiction.Hence $ |A \psi^2 A| \geq |A|. $\\
\textbf{Case III.3 :} $y=3$ \\
We have $A \psi A \psi A \psi A = \{(1,2,3,4),(1,2,3,5),(1,2,4,5),(1,3,4,5),(2,3,4,5)  \} .$\\ 
\textbf{Claim :} $|A \psi^3 A|=5.$\\
Consider the set $ F = \{1,2,3,4,5\}.$ We define $ F^* = \{ X \text{  } | \text{  } X \in 2^F , |X|=4 \}. $\\
Note that $|F^*|=5$. For any $X_i,X_j \in F^*$, we have that $ |X_i \cap X_j| = 3,$ by pigeonhole principle. Therefore any two elements in $A \psi ^3 A$ has to be distinct in modulo $6,$ or otherwise we have a contradiction by the above arguement. Hence proved.
\end{proof}
Our main theorems can be generalized to prove Gao and Thangadurai's conjecture $\ref{main-conjecture}$  for any even $N.$ The following remark is true. It is the generalization of the above lemma for even $N.$ 
\begin{remar}
\label{lp-sp-global}Let $N \in 2\mathbb{N}$ and let $A\subseteq  \mathbb{Z}_{N}.$ For any natural number $y < |A|,$
$$ | A \psi^y A | \geq |A|  \  \text{ if $1 \leq y < |A|-1$,} \
 | A \psi^y A | = 1 \  \text{if $y = |A|-1.$ }$$ 
\end{remar}
\begin{lem} 
\label{lp}Let $A,B \subseteq  \mathbb{Z}_{6}.$\\
(i) If $ |A|=4,$ then $|A \psi A| \geq 4$ and $A \psi A$ contains all odds.\\
(ii) If $ |A|=3$ and $|B
|=2,$ then $|A+B|\geq 3$ with equality iff $A, B$ are of form $\two.$ \\
(iii)If $|A|=3$, then $A \psi A$ contains even number of odds.\\
(iv)If $|A|=|B|=3,$ then $|A+B| \geq 3$ with equality iff $A,B$ are of form $\two.$ For $A=B$ we have $|A \psi A| = 3.$
\end{lem}
\begin{proof}
(i) We have $|A \psi A| \geq 4, $ hence $A$ must contain atleast one odd. If $A$ contains all odds, we are done. If $A$ contains all evens, then also we are done, since $|A|=4$ and number of evens is three, there must be one odd element in $A.$ Summing this element with the three evens, we obtain three distinct odd elements. Let us suppose $A$ has two odds and two evens. Let $ O = \left\lbrace a_1,a_2 \right\rbrace $ where $a_1,a_2$ both are odd and  $E = \left\lbrace b_1,b_2 \right\rbrace $ where $b_1,b_2$ both are even.  By applying lemma$~\ref{gp}$ with $n=3,$ we have $ |O+E|  \geq 3.$ But we also know $a_i+b_j$ is always odd for $1 \leq i \leq j \leq 2$. Therefore, we conclude $ O+E =  \left\lbrace 1,3,5 \right\rbrace .$ 
\\
(ii)
It is clear that if $A, B$ are of form $\Romannum{2},$ then $|A+B| = 3.$ Suppose $A$ is not of form $\Romannum{2},$ let
$ A =\left\lbrace a_1, \ a_2, \ a_3 \right\rbrace $ and $  B=\left\lbrace b_1, \ b_2 \right\rbrace.$ There exists a subset 
$A\textprime =\left\lbrace a_1,a_2\right\rbrace$ which is of form $\Romannum{2}.$ By applying lemma $~\ref{gp}$ with $n=3,$  
$|A\textprime + B| \geq 3$ as $A\textprime$ is not of form $\one.$  If $|A\textprime + B| \geq 4,$ then we are done. Suppose $|A\textprime + B|=3,$ we know that two distinct elements $a_1+b_1,$  $a_1+b_2$ belong to $A\textprime + B.$  Hence we have $a_1 + b_2 = a_2 +b_1.$  Consider $a_3+b_2.$  It cannot be $a_1+b_2,$ $a_2+b_2.$ If $a_3 + b_2 = a_1 +b_1,$ then 
$a_3-a_1=b_1-b_2=a_2-a_1$ implies that $a_3 = a_1,$ which is a contradiction. Hence $a_3+b_2$ is distinct from $A\textprime+B$  and $|A + B| \geq 4.$ It is clear that $|A + B| \geq 4$ if $A$ is of form $\two$ and $B$ is not of form $\two.$
\\
(iii)  Consider $ A \psi A =  \left\lbrace c_1= a_1+a_2,c_2 = a_2+a_3 ,c_3 =a_3+a_1 \right\rbrace. $ We have $c_1+c_2+c_3 = 2(a_1+a_2+a_3),$ an even number. It follows that two elements in $A$ have to be odd or none of them are odd. Hence $A \psi A$ has even number of odds.\\
(iv) If $A$ and $B$ both were of form $\two,$ then $A+B$ is also of from $\two.$ Suppose WLOG , $A$ is not of form $\two$ and let $B\textprime$ is a subset of $B$ with cardinality 2. Then by (ii),
 $|A+B\textprime| > 3$. But we know that $A+B\textprime \subseteq A+B.$ Hence we are done.
 \end{proof}

\begin{lem} 
\label{fourzero}
Let $a,b,c,d \in  \mathbb{Z}_6$ be distinct elements. Define
$$ P = \left(2(a+b+c),2(a+b+d),2(a+c+d),2(b+c+d) \right) \in \mathbb{Z}_6^4. $$ $$ Q = \left( 2(a+b)+c+d,2(a+c)+b+d,2(b+c)+a+d,2(a+d),b+c \right) \in \mathbb{Z}_6^4. $$ Then one of the following is true.\\
\textbf{(i)} Two zeroes in the entries of P
\textbf{(ii)} Two zeros in the entries of Q
\textbf{(iii)} One zero entry in P and Q.
\end{lem}
\begin{proof} Let $R = \left( a+b+c,a+b+d,a+c+d,b+c+d\right) \in \mathbb{Z}_6^4.$ All four entries of $R$ are distinct. Now if $ 0$ and $3$ both are entries of $R,$ then (i) is true.  If $0$ or $3$ is not an entry of $R,$ then $1$, $2,$ $3,$ and $4$ are entries of $R.$ So (ii) is true.  If either $0$ or $3$ is an entry of $R,$ then there are two entries in $R,$ which are inverses of each other. So (iii) is true.
\end{proof}
\begin{cor}
\label{maincor} Let $A = \left\lbrace a, b,c,d,e \right\rbrace$ be distinct elements of $\mathbb{Z}_6.$   Then $2(a+x+y) = 0$, for some $x,y \in A - \left\lbrace a \right\rbrace.$
\end{cor}
\begin{proof} Let $A \textprime = \left\lbrace b,c,d,e \right\rbrace.$
By lemma $~\ref{lp}$ (ii), $A\textprime \psi A\textprime$ contains all odds. Either $-a$ or $-a+3$ is odd. It follows that  $2(a+x+y) = 0$, for some $x,y \in A - \left\lbrace a \right\rbrace.$
\end{proof}

\begin{center}
\textbf{Case $35$ \ : \ 4+3+2+2+2}
\end{center} Let the row with 4 elements be $a,$ the one with three be $b,$ the one with 2 be $c,$ $d,$ $e,$ and $f$ be the value of the remaining row. If $f+3=a,$ then $2a+b+c+d+e=0.$ We know that $|A \psi A|=4$ and $|B+C+D+E| \geq 3,$  since $|B| =3.$  If $f+3=b,$ then $a+2b+c+d+e=0.$  We know that $|A+C+D+E| \geq 4$ and $|B\psi B| \geq 3.$  If $f+3=c,$ then we have $a+b+2c+d+e=0.$ We know that $|A| = 4$ and $|B+C \psi C+D+E| \geq 3,$ since $|B| =3.$  The proofs for the cases $f+3 = d$ or $e$ follows similarly.
\begin{center}
\textbf{Case  $36$ \ : \ 4+3+2+2+1+1}
\end{center}
Let the row with 4 elements be $a,$ the one with 3 be $b,$ the one with 2 be $c$ and $d,$ and the remaining be $e$ and $f.$ Let the set of second coordinates of the elements in row $a$ be $A,$ $b$ be $B$ and so on.\\
\textbf{Case I : } $a+3\neq b.$ WLOG, suppose $a+3=d$ or $a+3=e,$ then we have $2a+b+c+e+f=0$ or $2a+b+c+d+f=0$ respectively. We also know that $|A \psi A| \geq 4.$ The sets $B+C+E+F$ or $B+C+D+F$ contains atleast 3 elements because $|B| =3.$ Hence one of them is an inverse of an element in $A \psi A.$ Hence $A \psi A+B+C+E+F$ or $A \psi A+B+C+D+F$ contains zero.\\
\textbf{Case II :} $a+3=b.$ 
This implies that $c+3$ is either $d$ or $e$ WLOG. Hence either $a+b+2c+e+f=0$ or $a+b+2c+d+f=0$ respectively. $|A|=4$ and $B+C \psi C+E+F$ or $B+C \psi C+D+F$ has atleast three elements because $|B|=3.$ The inverse of an elements in the latter sets must belong in $A.$ 
\newpage 
\begin{center}
\textbf{Case $37$ \ : \ 4+2+2+2+2+1}
\end{center}
 Let $a,b,c,d,e,f$ denote the row value for the corresponding partition in the given order respectively. Let the set of second coordinates of the elements in rows be $A,B,C,D,E,F$ respectively. Since we have considered all rows, $ a+b+c+d+e+f \neq 0$ as these are distinct elements in $\ZZ_6.$  Assume $f = a+3.$ Then we have $ 2a+b+c+d+e = 0.$ Let $A = \left\lbrace \alpha_1,\alpha_2,\alpha_3,\alpha_4 \right\rbrace,
   B = \left\lbrace \beta_1,\beta_2 \right\rbrace,
   C = \left\lbrace \gamma_1,\gamma_2 \right\rbrace,
  D = \left\lbrace \delta_1,\delta_2 \right\rbrace,\\
   E = \left\lbrace \epsilon_1,\epsilon_2 \right\rbrace \ \mathrm{and} \ F = \left\lbrace \zeta \right\rbrace$.\\
\textbf{Case I :} At least one of the rows $B,C,D,E$ is not of form $\one.$ WLOG, let $B$ be not of form $\one.$ Consider the set $ A \psi A +B+C+D+E $. Then it follows from Lemma $\ref{gp}$ that $ |B + C| \geq 3 ,$ so $ |B+C+D+E | \geq 3$.  We also have $|A \psi A| \geq 4$ by Lemma \ref{lp}, Since $ |B+C+D+E| +|A \psi A| \geq 7 > 6  .$
\textbf{Case II :}
The rows $B,C,D,E$ are of form $\one.$
By Lemma $\ref{gp},$ $ |B+C+D+E| = 2. $ Let $B+C+D+E = \left\lbrace \beta+\gamma+\delta+\epsilon,\beta+\gamma+\delta+\epsilon+3 \right\rbrace$. Therefore $B+C+D+E$ contains an odd. By Lemma $\ref{lp}$ (ii), $ A \psi A$ contains all odds and therefore the inverse of the odd element in $B+C+D+E$ exists in $A \psi A.$ 
Now $a+3$ can be either $b,c,d$ or $e.$ WLOG assume $e=a+3,$  we have $ 2a+b+c+d+f = 0.$ The proof is similar by relplacing $E$ with $F$ as above.

\begin{center}
 \textbf{Case $31$ \ : \ 4+4+2+1+1+1}
 \end{center}
\textbf {Case I : } We have $2a+b+c+d+e = 0.$ By Lemma 4.5(i) we have
$|A\psi A| \geq 4.$ Also, $|B+C+D+E| \geq 4.$\\
\textbf {Case II : } $f=c+3.$ We have $2c+a+b+d+e=0.$ Also, $|A|=4$ and
$|C\psi C+B+D+E| \geq 4.$
\textbf {Case III : } $f=e+3.$\\
\textbf{sub case III.1 :} $b=a+3.$ WLOG, let $a=f+1, \ b= f+4, \ c=f+2, \ d=f+5.$\\
Now $2(b+c)+a+d =0,$ $|B \psi B + C \psi C| \geq 4$ and $|A+D| \geq 4.$
\\
\begin{center}
\textbf{Case $30$ \ : \ 4+4+2+2+1}
\end{center}
Let  the rows with 4 elements be $a$ and $b,$ the rows with 2 elements be $c$ and $d,$ the one with single element be $e$ and $f$ be the remaining row.\\ 
{\textbf Case I : $f+3 \neq e$.}
WLOG, assume $f+3=a$ or $f+3=c.$ We have $2a+b+c+d+e=0,$ $|A \psi A| = 4,$ $|B|=4$  or $a+b+2c+d+e=0,$ $|A|=4$ and $|B|=4.$ 
{\textbf Case II : $f+3=e.$} 
{\textbf Case II.1: $a=b+3$}\\
If $a=b+3$ and $f=e+3$, we know that $c=d+3$. We claim that $3a+b+c+e$ or $3a+b+d+e$ is $0.$ The difference  is $c-d$ which is 3. The sum is $2b+c+d+2e=2b+2c+2e+3$. Since $2b, 2c, 2e$ are different congruences modulo 6 and all of them are even, they have to be $0, 2, 4$ in some order. Hence their sum is zero modulo 6. It follows that $2b+2c+2e+3 = 3.$ This proves our claim as if the sum and difference of two numbers is $\equiv 0 \pmod{3}$, then the numbers are three and zero.
WLOG, let $3a+b+c+e = 0.$ Since $|A\psi A\psi A| =4,$ and $|B|=4,$ we are done.\\
{\textbf Case II.2 : $a \neq b+3$}\\
Consider $3a+2b$ and $2a+3b.$ They are distinct and one of them has inverse in $\lbrace a,b,c,d,e \rbrace.$ Suppose WLOG that $3a+2b$ has its inverse in $\lbrace a,b,c,d,e \rbrace.$ It cannot be $'a'$ as that would imply that $4a+2b=0$ or equivalently, $a=b+3$ which is a contradiction. If it is $'b',$ then $3a+3b=0.$ Then $ 0 \in A \psi A\psi A+B\psi B\psi B,$ since $|A\psi A\psi A| = 4$ and $|B\psi B\psi B| = 4.$  If it not $b,$ we would choose three from row $a,$ two from $b,$ and one from someother row $x.$
The second coordinates would be from $A \psi A\psi A+B\psi B+ X$ if $X$ is the set of second coordinates of the $x$-th row. We have $|A\psi A\psi A|=4,$ and $|B\psi B+X| \geq 4.$ 
\begin{center}
 \textbf{Case $41$ \ : \ 3+3+2+2+2+1}
 \end{center}	
\textbf {Case I : } $c=d+3.$ We have $2c+f+a+b+e =0.$ \\
\textbf {Case I.1 :} $A$ or $C$ are not of form $\two.$\\
\textbf {Case I.1a :} $A$ is not of form $\two.$ Then $|A+E| \geq 4$ by Lemma \ref{lp}. Since $| C \psi C+F|=1,$ we have
 $|A+E+C \psi C+F| \geq 4.$ Since $|B|=3,$ by Lemma $\ref{main},$  $0 \in C \psi C+F +A+B+E. $ The same proof applies if $E$ or $D$ are not form $\two.$ The case when all rows are of form $\one$ or $\three,$ follows the same way. \\
\textbf {Case I.1b :} $C$ is not form $\two.$ If $e=a+3,$ $2a+b+c+d+f = 0.$ Then $|A \psi A| \geq 3,$ and $|B +C| \geq 4,$ therefore 
$0 \in A\psi A+B+C+D+F.$ Otherwise, $e=f+3,$ $2e+a+b+c+d = 0.$ Then $|E \psi E + A| \geq 3,$ and $|B+C| \geq 4,$ therefore
$0 \in E \psi E +A+B+C+D.$\\
{\textbf Case I.2 :} All rows are form $\two$. Consider $a,b,c,d,e.$ From Corollory $\ref{maincor},$ we know that $2(a+x+y)=0$ for some $x,y \in \left\lbrace b,c,d,e \right\rbrace.$ Let $X$ and $Y$ be the second coordinates of the rows $x, y$ respectively. $A \psi A$ contains all evens and $X \psi X,$ $Y  \psi Y$ contains evens, therefore $ 0 \in A\psi  A + X \psi X + Y \psi Y.$\\
\textbf {Case II : $c=f+3.$} We have $2c+a+b+e+d=0.$ \\
\textbf {Case II.1 :} Let $D$(or $E$) be not of form $\two.$ Then  $ |B+D| \geq 4 $ (or  $ |B+E| \geq 4 $  by Lemma $\ref{lp} (ii).)$ Also $ |A+E| \geq 3, $ which solves for this particular case, since $ C \psi C$ contributes only one element.\\
\textbf {Case II.2 :} If $C,$ $D$ and $E$ are form $\two,$ then at least one of them is zero $(2a+c+d),2(b+c+d),2(a+d+e),2(b+d+e)$  by Corollory $\ref{maincor}.$  We also have $A \psi A$ and $B \psi B$ contain all evens and $ C \psi C , D \psi D $ are subset of evens, which solves this case.\\
\textbf {Case II.3 :} $C$ is not form $\two.$\\
 \textbf {Case II.3a :} If $a=e+3,$ then $2a+b+c+d+f = 0.$ We have $|B+C| \geq 4,$ and $|A \psi A +D+F| \geq 3.$ Therefore, we have $0 \in A \psi A +B+C+D+F.$\\
 \textbf {Case II.3b :} If $d=e+3,$ then $2d+a+b+c+f = 0.$ We have $|B+C| \geq 4,$ and $|A+D \psi D +F| \geq 3.$ Therefore, we have $0 \in D \psi D +A +B+C+F.$\\
\textbf {Case III : $c=b+3$}\\
\textbf{ Case III.1 :}  $A,C,D,E$ are not form $\two.$ WLOG $C$ is not of form $\two.$ We have two subcases\\ 
\textbf {Case III.1a :} If $d=a+3,$ then $2a+b+c++e+f = 0$ and $|A+C| \geq 4$ and $|B| = 3.$ solves this case.\\
\textbf {Case III.1b :} If $f=a+3,$ it follows similarly.\\ 
\textbf {Case III.2 :} If all rows were of form $\two.$ Consider $a,b,c,d,e.$ From Corollory $\ref{maincor},$ we know that $2(a+x+y)=0$ for some $x,y \in \left\lbrace b,c,d,e \right\rbrace.$ $A \psi A$ contains all evens and $X \psi X,$ $Y  \psi Y$ contains evens, for $X,Y \in \left\lbrace B,C,D,E \right\rbrace$ and therefore $ 0 \in A\psi  A + X \psi X + Y \psi Y.$
\begin{center}
\textbf{Case $42$ \ : \ 3+2+2+2+2+2}
\end{center}
Assume WLOG $f=a+3$ and we have the first coordinate $2a+b+c+d+e=0.$     
{\textbf Case I :} There exists  atleast one row with form $\one$ and there exists atleast one row with form $\two.$ WLOG let row $B$ be form $\one$ and row $C$ be form $\two.$  By Lemma $\ref{lp} (ii),$ we have $ |A \psi A + B| \geq 4 .$ Then by Lemma $\ref{gp},$ $ |C+D| \geq 3$ and therefore $ |C+D+E| \geq 3 .$ Hence by Lemma $\ref{main},$  $0 \in \ A \psi A +B+C+D+E.$
{\textbf Case II:} All the rows are of form $\two.$  
By Lemma $\ref{lp} (iii),$ we have $ A \psi A$ contains even number of odds. Suppose it contains two odds and an even. Since $A \psi A$ is not of form $\two,$ we have by Lemma $\ref{lp} (ii),$ $ | A \psi A + B | \geq 4 $ and from Lemma $\ref{gp},$ we have $ | C + D | \geq 3 $ and  $ | C + D + E | \geq 3 $
which solves one part. Suppose $ A \psi A$ contains only evens. Then WLOG consider $ b=a+1,c=a+2,d=a+4,e=a+5.$ We have $ 2a+2b+2e = 0.$  We have $B \psi B + E \psi E$ to be an even number and  since $A \psi A$ contains all evens, we are done.\\
{\textbf Case III :} Rows $B,C,D,E$ are of form $\one.$\\
Let $A = \left\lbrace \alpha_1,\alpha_2,\alpha_3 \right\rbrace,
   B = \left\lbrace \beta,\beta+3 \right\rbrace,
   C = \left\lbrace \gamma,\gamma+3 \right\rbrace,
  D = \left\lbrace \delta,\delta+3 \right\rbrace,
   E = \left\lbrace \epsilon,\epsilon+3 \right\rbrace , F = \left\lbrace \zeta_1,\zeta_2 \right\rbrace.$ \\
   Since $f=a+3$,  WLOG, assume $b=a+1,c=a+2,d=a+4,e=a+5.$ Hence $2(a+b+e),2(a+c+d),2(a+f)+b+e,2(a+f)+c+d,2(b+e)+c+d,2(c+d)+b+e$ are $0.$ Let
 $$ Y =  B \psi B + E \psi E \cup  C \psi C + D \psi D \cup
 F \psi F + B + E \cup F \psi F + C + D  \cup B \psi B + E \psi E + C + D \cup C \psi C + D \psi D + B + E,$$
$$ X = \left\lbrace 2\beta+2\epsilon , 2\gamma+2\delta , \zeta+\beta+\epsilon, \zeta+\beta+\epsilon+3 , \zeta+\gamma+\delta,\zeta+\gamma+\delta+3 \right\rbrace,$$
$$ Z =      
    \left\lbrace 2\beta + 2 \epsilon + \gamma + \delta , 2\beta + 2 \epsilon + \gamma + \delta +3 ,
     2\gamma + 2 \delta  + \beta + \epsilon , 2\gamma + 2 \delta  + \beta + \epsilon+3 \right\rbrace,$$
where $ \zeta = \zeta_1 + \zeta_2.$
Define $ W = A \psi A + X.$ Note $ Y = W \cup Z.$
If $0 \in  Y$, then we are done. Suppose zero isn't in $Y,$ then  $\beta+\epsilon \neq \gamma+\delta $ and $\beta+\epsilon \neq \gamma+\delta+3 ,$ else  $ 0 \in Z.$ Hence we have $\zeta+\beta+\epsilon, \zeta+\gamma+\delta,\zeta+\gamma+\delta+3$  and $\zeta+\beta+\epsilon+3$ are distinct modulo 6. Hence we have $|X| \geq 4. $  By Lemma $\ref{lp} (iv),$ we have $ |A \psi A| = 3.$ Therefore $|A \psi A| + |X| = 7 >6$ and it follows from Lemma $\ref{main},$  $0 \in A \psi A + X = W,$ which is a contradiction. 

\begin{center}
\textbf{ Case $40$ \ : \ 3+3+3+2+1+1}
\end{center}
 	\textbf{Case I : $ d = c+3$} We have $2c+a+b+e+f = 0. $ Consider the set $ C \psi C +A+B+E+F.$\\ 
 	\textbf{Case I.1 :}  If at least one of $A,B,C$ were not of form $\two,$ then WLOG , $A$ is not form $\two,$ we have $|A + B| \geq 4 $ , by Lemma $~\ref{lp}$ (iv) and $ |C \psi C| = 3.$ Therefore by Lemma $\ref{main},$ it follows that $0 \in C \psi C +A+B+E+F.$ \\
\textbf{Case I.2 :} Suppose $A,B,C$ were of form $\two$ and $D$ is with different parity, we have two sub cases \\
 	\textbf{subcase I.2.1 :} $a+3 = b$ \\
In this case, we have $2a+c+d+e+f = 0 $  By Lemma 4, we have $ |C+D| \geq 4 $ and we have $|A \psi A + E+F| = 3.$ By lemma $\ref{main},$ it follows that $0 \in A \psi A +C+D+E+F.$ \\
 	\textbf{subcase I.2.2 :} $a+3=e$ \\
 	We have the same proof as above, by replacing $e$ with $b$ and $E$ with $B.$\\
 	\textbf{Case I.3 :} Now if $A,B,C,D$ all were of form $\two.$
Let $T = \left\lbrace a,b,c,d \right\rbrace$ and suppose $T \neq \left\lbrace a'+1,a'+2,a'+4,a'+5 \right\rbrace$ for some $a' \in \mathbb{Z}_6$. Then by Lemma $\ref{fourzero},$ we have $2(x+y+z) $ for some $x,y,z \in T$. Consider the set  $X \psi X + Y \psi Y + Z \psi Z.$ Since all rows are of form $\two,$ we have $P \psi P$ contains all evens for $P =A,B,C.$ Hence by Lemma $\ref{main},$ we conclude that there exists $ 0 \in X \psi X + Y \psi Y + Z \psi Z.$ Therefore  assume that $a=a'+1,b=a'+4,c=a'+2,d=a'+5$ for some $a' \in \mathbb{Z}_6$ (If $a,b,c,d$ are in different order from $\{ a'+1,b=a'+4,c=a'+2,d=a'+5 \},$ then the same proof works). WLOG let $e=a', f=a'+3.$ \textbf{Case I.3.1 :} Suppose $B,C$ were of the same parity type. We have $2(a+d)+b+c = 0.$  We have $A \psi A , D \psi D$ and $B+C$ contain all evens and hence $0 \in A \psi A + D \psi D + B + C.$ Similarly it holds true when $A$ and $C$ were of same parity type. \\
Hence we assume that $B$ and $C$ are of different parity  type and $A$ and $D$ are of different parity types.\\
\textbf{Case I.3.2 :} Suppose $E$ is even and $F$ is odd, then we have the following sub cases \\
\textbf{Sub case I.3.2a :} $A,B$ are of same parity type. \\ By assumption, $B+C$ contains all odds (since $|B+C|=3,$ by lemma $\ref{lp} (iv)).$ If $A$ is even parity type, then $3a+b+c+f=0.$ Since $A$ is even, $A \psi A \psi A$ is also an even number and  $B+C+F$ contains all evens, it follows that $0 \in A \psi A \psi A + B + C + F.$ Now if $A$ is odd parity type. Therefore $B$ is of odd parity type by assumption. Then we have $3b+a+d+e = 0.$  We have $A+D$ to contain all odds. Since $ B \psi B \psi B$ is odd and $A+D+E$ contains all odds, we have $ 0 \in  B \psi B \psi B + A + D + E.$ \\
\textbf{Sub case I.3.2b :} $A,C$ have the same parity type.\\
 We have $B$ and
$D$ are of same parity type. Therefore $B+D$ contains all evens and $A+D$ contains all odds. If $A$ is even parity type, then $B$ is of odd parity type. We have $3b+a+d+e=0$ and the set of possible second coordinates is $B \psi B \psi B+ A+D+E$. But $B \psi B \psi B$ is also odd and since $A+D+E$ contains all odds, we are done. Now if $A$ is odd parity type. Then we have $3c+a+d+e = 0.$ Since $ C \psi C \psi C$ is odd and $A+D+E$ contains all odds,  $ 0 \in C \psi C \psi C + A + D + E.$. \\
\textbf{Case I.3.3 : } Suppose $E$ is odd and $F$ is even.\\
\textbf{Sub case I.3.3a :} $A,B$ are of same parity type. \\ By assumption, $C$ and $D$ are of same parity type and therefore $A+D$ contains all odds. If $A$ is even parity type, then $B$ has even parity too. We have $3b+a+d+e=0.$   Since $B$ is even parity type, $B \psi B \psi B$ is also even and since $A+D+E$ contains all evens, it follows that $0 \in \ B \psi B \psi B + A+D+E.$  Now if $A$ is odd parity type, $C$ is of even parity. Then we have $3c+a+d+e = 0,$ and $ C \psi C \psi C$ is an even number and $A+D+E$ contains all evens, hence $0 \in C \psi C \psi C + A + D + E.$\\
\textbf{Sub case I.3.3b :} $A,C$ are of same parity type. \\ By assumption, $B$ and $D$ are of same parity type. Therefore $B+D$ contains all evens and $A+D$ contains all odds. If $A$ is even parity type, then we have $3c+b+d+f=0,$  $C \psi C \psi c$ is an even number and since $B+D+F$ contains all evens, it follows that $0 \in C \psi C \psi C + B+D+F.$ Now if $A$ is odd parity type, then $B$ is of even parity type. We have $3c+a+d+e = 0,$   $ C\psi C\psi C$ is an odd number  and $A+D+E$ contains all odds and so $ C \psi C \psi C + A + D + E$ contains all evens. Hence $ 0 \in C \psi C \psi C + A + D + E.$ \\
\textbf{Case I.3.4 :} If $E+F$ is an even number\\
\textbf{Sub case I.3.4a :} $A,B$ is of same parity.\\
If $A$ is  even parity type, we have $2c+a+b+e+f=0.$ Since, $C$ is odd parity type, $C\psi C$ is even parity type and $A+B+E+F$ contains all evens. Hence $ 0 \in 2^C+A+B+E+F$. If $A$ is odd parity type, $D$ and $C$ are of even parity type. We have $2c+a+b+e+f=0.$ Since $C \psi C$ is of even parity and $A+B$ is even parity implies $A+B+E+F$ is of even parity too, we are done.\\
\textbf{Sub case I.3.4b :} $A,C$ is of same parity.\\
We consider the following further sub cases,
$E$ is even, $F$ is even 
We have the same proof as the subcase $I.3.2b$, since we have only used the condition $E$ is even.\\
$E$ is odd, $F$ is odd
If $A$ is of even parity type, then $B$ is odd parity. The first coordinate $3a+b+c+f$ is $0.$ We know that $A\psi^2 A$ is an even number and hence $A\psi^2 A+C$ is even parity type. Also, $B+F$ is even parity type, we are done.
If $A$ is odd parity type, then $B$ and $D$ are even parity types and therefore $B+D$ is even parity type. We have $3c+b+d+f = 0.$ 
 We know $C$ is an odd parity type, implies $C \psi^2 C $ is an odd number and hence $C \psi^2 C +F$ is an odd number. Hence $ 0 \in C \psi^2 C + B+D+F.$ \\
\textbf{Case II : $d = f+3$}
The row values of $A,B,C,E$ have to be $f+1,f+2,f-1,f-2$ in some order. There must be two among $a,b,c$ of the form $f+x$ and $f-x$ where $x = 1$ or $2.$  WLOG let $a=f+x, b=f-x$. \\
\textbf{Case II.1 :} If all sets $A,B,C,D$ are form $\two.$ Then we can verify that $2d+2a+2b = 0.$ Since we know that $A \psi A$, $ B \psi B$ contain all evens and $D \psi D$ is an even number. Therefore $D \psi D + A \psi A + B \psi B$ contains all evens and hence contains zero.\\
\textbf{Case II.2 :}  At least one of $A,B,C,D$ is not of form $\two.$
 We consider the following. $a+3$ can't be $b$. Hence $a+3=c$ or $a+3=e.$  We have $2a+b+d+e+f=0$ or $2a+b+c+d+f=0.$
 If $A$ is not form B, so is $A\psi A$. Therefore, $|A\psi A+D|\geq4.$ Since $|B|=3,$ by Lemma $\ref{main},$  $ 0 \in A\psi A+B+D+E+F$ or $A\psi A+B+C+D+F.$
 
 If $B$ or $D$ is not form $\two,$ then $|B+D|\geq 4$ and $|A\psi A|=3.$ Hence both the sets must contain all elements of $\ZZ_6.$ 
 If $C$ is not form B, $|B+C|\geq 4$ and $|A\psi A|=3.$  Hence, if $a+3=e$, $ 0 \in A\psi A+B+C+D+F,$  by Lemma $\ref{main}.$ If $a+3=c$, then $b+3=e.$ We have $2b+a+c+d+f=0$. The second coordinates can be only in $B\psi B+A+C+D+F$. However, 
 $|A|=3$ and $|C+D|\geq4,$ by Lemma $\ref{lp}, (iv).$ Therefore $0 \in B\psi B+A+C+D+F.$\\
\textbf{Note:} The cases we have handled so far assumed the distributions happen in rows. However same proofs hold for columns as well.
\begin{center}
\textbf{Case $38$ \ : \ 3+3+3+3+1}
\end{center}
The rows with 3 elements be values $a, b, c$ and $d,$ and the one with single element be $e.$ Let the row that was not chosen be $f.$ 
\textbf{Case I : $A,B,C,D$ are of form $\two.$ }\\
WLOG let the elements of  $A,B,C,D$ be even. Let it be $$(a,0),(a,2),(a,4),(b,0),(b,2),(b,4),(c,0),(c,2),(c,4),(d,0),(d,2),(d,4).$$ Look at the column distribution. It is distributed in the following way. Four elements each in columns $0,2$ and $4$, that is, $4+4+4$. Hence we are done by $3+4+4$ case.
\\
We have $A,B,C,D$ are of form $\two.$ Let three among $A,B,C,D$ are of one parity and the fourth is of another parity. Then we again look at the column distribution. The $12$ elements are arranged in $3+3+3+1+1+1$ configuration. Hence the thirteen elements can either be in $4+3+3+1+1+1$ or $3+3+3+2+1+1$ both of these cases, we have already shown to contain a zero six sum.
\\
If two rows have second coordinates of one parity and the rest have second coordinates of another, then looking at the column distribution, it is $2+2+2+2+2+2$ distribution. The $13$ elements will go in some column making the column distribution $3+2+2+2+2+2,$ which we have proved already.\\
{\textbf Case II : } At least one of $A,B,C,D$ is not of form $\two.$
WLOG Let $A$ be not of form $\two.$\\ 
{\textbf Case II.1 : $e=f+3$}\\
If $e=f+3$, then WLOG $a=f+1, b= f+2, c=f+4$ and $d=f+5.$ Hence $2(b+c) + a+d =0.$ Then by Lemma $\ref{lp} (iv),$ $|A+D| \geq 4,$ and $|B \psi B+C \psi C| \geq 3.$ Hence $0 \in B \psi B+C \psi C +A +D.$\\ 
{\textbf{Case II.2 : $e \neq f+3$}\\
If $e\neq f+3,$  then $f+3 =a $  or $f+3=b.$ Then $2a+b+c+d+e=0$ or $a+2b+c+d+e=0.$  For the first case, $|A\psi A+B| \geq 4,$ and $|C+D+E| \geq 3.$ Therefore by Lemma $\ref{main},$ $0 \in A \psi A+B+C+D+E.$ For the second case, $|A+B\psi B| \geq 4,$
and $|C+D+E| \geq 3.$ Therefore $0 \in A+B\psi B+C+D+E.$ 
 \begin{center} 
 \textbf{Case $39$ \ : \ 3+3+3+2+2}
 \end{center}
\textbf{Case I}: f =a+3.
We have 2a+b+c+d+e=0. Consider $A \psi A+B+C+D+E$.\\
\textbf{Case I.I}: If atleast one is not of form $\two$\\ WLOG, A is not of form $\two$, we have $|A \psi A| =3 $, and also $A \psi A$ is not of form $\two$. Therefore, it's clear that $|A \psi A + B| \geq 4$ and $|C+D+E| \geq 3$. Therefore by Lemma $\ref{main},$ it follows that $0 \in A \psi A+B+C+D+E $. \\
\textbf{Case I.II} All are of form $\two$.\\
Let $A,B,C$ be of one parity type (say even) and $D,E$ be of another parity type(odd). Then $ A \psi A +B+C$ contains all evens and $D+E$ contains two even numbers. Hence, we have a zero in this case.
Suppose all are of same parity type. Then, it's column wise distribution looks like $5+5+3$, which follows from the case $3+3+4$. \\
Suppose $A,B,C$ are of the same parity type and $D,E$ are of different parity type. Then it's column wise distribution looks like $4+4+3+1+1$, which follows from the case $3+3+4$. \\
Suppose two in $A,B,C$ are of the same parity type and $D,E$ are of same parity type.  Then the column distributions will look like $4+4+2+1+1+1$ or $4+3+3+1+1+1$ or $3+3+2+2+2+1$ or $3+2+2+2+2+2$, all which we have showed that there exists a zero six sum.\\
Suppose two in $A,B,C$ are of the same parity type and $D,E$ are of different parity type. Then it's column wise distribution looks like $3+3+2+2+2+1$, which we have proved already.\\
\textbf{Case II} f=d+3. We have $2d+a+b+c+e=0$, suppose atleast one of $A,B,C,D,E$ be not of form $\two$. 
WLOG let $D$ be not of form $\two$. Then WLOG let $a,b,c,e$ take values $f+1,f+2,f+4,f+5$. If one of $A,B,C,E $ are not of form $\two$, then we have $ 0 \in D \psi D+A+B+C+E$. Therefore, assume all are of form $\two$. Now, if all $A,B,C,E$ are of same parity type, then $2(b+c)+a+e=0$ implies that $B \psi B +C \psi C $ contains all evens and $A+E$ is an even number and hence we have $ 0 \in B \psi B +C \psi C+A+E$. Suppose, one of $A,B,C,E$ is of different parity type. Then $3b+d+e+c = 0$ or $3c+d+a+b=0$. Since, $D$ contains an even and an odd, we can say that $0 \in B \psi B \psi B +D+E+C$ or $0 \in C \psi C \psi C +D+A+B$. Suppose all are of form $\two$.
If $A,B,C$ are of different parity type (say even) from $D,E$(say odd). Then, considering $3b+d+e+c=0$, we have that $B \psi B \psi B+ C$ contains all evens and $D+E$ contains atleast one even. Hence, we have $ 0 \in B \psi B \psi B+ D+E+C.$
For the rest of the cases, we have the same proof as in the case $f=a+3$, where all were of form $\two.$
\section*{Acknowledgements}
We would like to thank IISER, TVM for providing excellent working conditions. 
We thank T. Kathiravan and Pasupulati Sunil Kumar for their careful reading of this paper.

\end{document}